\documentclass[a4paper,12pt,oneside]{article}
\author{Asl{\i} Deniz} 
\title{A Landing Theorem for Periodic Dynamic Rays for Transcendental Entire Maps with Bounded Post-Singular Set}
\date{}
\usepackage[english]{babel}
\usepackage{graphicx}
\usepackage{amscd}
\usepackage{amsfonts}
\usepackage{amscd}
\usepackage{latexsym}
\usepackage{epsfig}
\usepackage{amssymb,amsmath}
\usepackage[usenames]{color}
\usepackage{amscd}
\usepackage{psfrag}
\usepackage{amsthm}
\usepackage[linktocpage=true]{hyperref}
\theoremstyle{plain}

\newtheorem{thm}{Theorem}
\numberwithin{thm}{section}
\newtheorem{prop}[thm]{Proposition}
\newtheorem{lem}[thm]{Lemma}

\newtheorem*{thmmain}{Main Theorem}
\theoremstyle{definition}

\theoremstyle{definition}

\newtheorem{rem}[thm]{Remark}

\begin{document}
\maketitle
\begin{abstract}
In this article, we present a landing theorem for periodic dynamic rays for transcendental entire maps which have bounded post-singular sets, by using standard hyperbolic geometry results.
\end{abstract}

\section{Introduction}
We consider the dynamical system given by the iterates of an entire map $f: \mathbb{C}\to \mathbb{C}$ which we denote by $f^n=f \circ \overset{n}{\cdots} \circ f$. Our interest is to understand the long term behavior of the orbits under $f$, in terms of their initial condition. The dynamical plane splits into two totally invariant sets; the set of points with stable behavior, i.e., the domain of normality of the iterates $\{f^n\}_{n}$, and its complement, i.e., the set of points with nonstable behavior. These two sets are called the \textit{Fatou set} and the  \textit{Julia set}, respectively. Understanding the topology of these complementary sets is one of the main goals in the field. Another dynamically relevant set is the \textit{ escaping set}, which consists of orbits that escape to infinity under iteration of $f$.  Depending on the nature of the entire map (i.e., whether a polynomial, or transcendental), its escaping set is associated to either the Fatou set, or the Julia set.

In this paper we study some special objects known as "dynamic rays" which play an essential role in studying the topological structure of the Julia set. Roughly speaking, dynamic rays are unbounded curves formed by orbits which escape to infinity following a given pattern, that is, according to some symbolic dynamics structure. In polynomial dynamics, dynamic rays are part of the Fatou set, since the point at infinity is superattracting, and belongs to the Fatou set. Instead, in transcendental dynamics, infinity is an essential singularity, and hence dynamic rays belong to the Julia set.  To start with, we look at dynamic rays for polynomials  following \cite{douhub1984}, where the concept of rays was first introduced. Suppose $f$ is a monic polynomial of degree $d\geq 2$. Then the point at infinity is a superattracting fixed point with multiplicity $d-1$.  Let $\mathcal{K}(f)$ denote the \textit{filled-in Julia set} of $f$, i.e., the set of all points with bounded orbit. If $\mathcal{K}(f)$ is connected, or equivalently, if  $\mathcal{K}(f)$ contains all finite zeroes of $f'$ (known as \textit{critical points} of $f$), then the locally defined  B\"ottcher map -which conjugates the dynamics near $\infty$ to $z\mapsto z^d$ near $\infty$- extends to a conformal isomorphism $\phi:\mathbb{C}\backslash \mathcal{K}(f)\rightarrow \mathbb{C}\backslash\overline{\mathbb{D}}$. \textit{A dynamic ray of argument $\theta$} is defined by the inverse image  $\phi^{-1}\{r e^{2\pi i\theta},r>1\}$, which we denote by $R_{f}(\theta)$.

A dynamic ray $R_{f}(\theta)$ is $k$-periodic, if there exists a minimal number $k\geq 1$  such that $f^{k}(R_{f}(\theta))\subset R_{f}(\theta)$, or equivalently $d^{q}\theta\equiv \theta\pmod{1}$. We say that $R_{f}(\theta)$  lands at some point if $\overline{R_{f}(\theta)}\backslash R_{f}(\theta)$ is a singleton. The famous landing theorem by Douady and Hubbard 
(see \cite[Chpt 8, Prop 8.4]{douhub1984}) states that for a polynomial of degree $d\geq 2$, if $\mathcal{K}(f)$ is connected, then every periodic dynamic ray lands at a periodic point, which is either repelling or parabolic.

In transcendental entire dynamics,  the point at infinity is not suitable for constructing rays in the same way as for polynomials. However,  for a large class of transcendental maps, there exist curves in the Julia set, consisting of points escaping to $\infty$ under forward iterates. This structure is considered to be analogous to the dynamic rays in polynomial dynamics. Because of the similarity, those curves are also called dynamic rays, or sometimes hairs. Many studies have been carried out to prove existence theorems for different classes of transcendental entire functions. Devaney and Krych in 1984 discovered that for exponential functions $z\mapsto \lambda e^z$ for $\lambda\in(0,1/e)$, the escaping set consists of curves tending to $\infty$ (\cite{devkry1984}).  This study was generalized by Devaney et. al. to exponential maps satisfying certain conditions (\cite{devgoldhub1999}), and was finally completed for arbitrary maps in this family in \cite{schzim2003a}.

The investigation of dynamic rays in larger classes of transcendental maps started in 1986 by Devaney and Tangerman  (\cite{devtang1986}). They observed the ray structure in classes of transcendental maps having some similarities with the exponential family. Later, Baranski  in \cite{bar2007}, proved the existence of dynamics rays for hyperbolic  maps of finite order which have one completely invariant Fatou component. A recent work \cite{rrrs2011} by Rottenfusser, R{\"u}ckert, Rempe and Schleicher  covers all the previous results as well as gives a proof of existence of rays in a more general class of maps. According to their work, dynamic rays exist for a class of transcendental entire maps which include those of finite order with bounded singular set, or finite composition of such maps. In the present paper, we work in this class, and denote it by $\widehat{\mathcal{B}}$.

Dynamic rays in transcendental dynamics are organized according to some symbolic dynamics. We do not consider this analysis here. For details about the existence and the uniqueness see \cite{rrrs2011}, or \cite{fagben2011}. We denote a dynamic ray by a continuous parametrized unbounded curve
\begin{eqnarray*}
g:(0,\infty)&\rightarrow&\mathbb{C}\\
t&\mapsto& g(t),
\end{eqnarray*}
and we call the parameter $t$, the \textit{potential}. In order to associate the potential to the dynamics, a \textit{ray model dynamics} $F:[0,\infty)\rightarrow[0,\infty)$ is used, which is a continuous, unbounded  and strictly increasing function, and which respects the dynamics, i.e.
$f(g(t))=\tilde{g}(F(t))$ for some ray $\tilde{g}$ and for all $t>0$.  A dynamic ray $g$ is $k$-periodic if $f^k(g)\subset g$. In particular, for a $k$-periodic dynamic ray $g$
\begin{equation*}
f^k(g(t))=g(F^k(t))
\end{equation*}
is satisfied. A \textit{fundamental segment} of a $k$-periodic dynamic ray is any ray segment with endpoints $g(t)$ and $g(F^k(t))$ for $t\in(0,\infty)$, and we denote it by $g[t,F^k(t)]$.

We say that $g$ lands if $\displaystyle{\lim_{t\rightarrow 0}g(t)}$ exists. As for polynomials, understanding the landing behaviors of dynamic rays is essential to understand the topology of the Julia set. Naturally, the study of landing behaviors of dynamic rays of transcendental maps  started with the exponential dynamics. Schleicher and Zimmer showed that in the exponential family,  every periodic dynamic ray lands if the singular orbit is bounded  (\cite{schzim2003a}). In \cite{rem2006}, Rempe showed that if the singular orbit of an exponential map does not escape to $\infty$, all periodic dynamic rays land.

In this article, we aim to prove a landing theorem for periodic dynamic rays for a general class of transcendental entire maps inside class  $\widehat{\mathcal{B}}$, where the dynamic rays are well-defined. In the proof, we use hyperbolic geometry, which is the  classical tool to deal with the landing properties of periodic dynamic rays. For a given entire map $f$, we define the \textit{set of singular values of $f$}, denoted by $\mathcal{S}:=\mathcal{S}(f)$, as the union of its \textit{critical values} ($f( c)$ where $f'( c)=0$),  its \textit{asymptotic values}  (i.e. values $a\in\mathbb{C}$ for which there exists a curve tending to infinity whose image lands at $a$), and their accumulation points. Define the post-singular set by
\begin{equation*}
\mathcal{P}=\overline{\cup_{n\geq 0} f^n(\mathcal{S})}.
\end{equation*}
We now state our main result.
\begin{thmmain}\label{landingtheoremforexternaldynamicrays} For $f\in\widehat{\mathcal{B}}$ with bounded post-singular set $\mathcal{P}$, all periodic dynamic rays land, and the landing points are either repelling or parabolic periodic points.
\end{thmmain}

The result is not new. It is implied by \cite[Thm B.1, Cor B.4]{rem2008}, where  the proof also uses hyperbolic geometry (see also \cite[Cor 2.19]{hel2009}). Our proof mainly differs by our approach to dealing with the following two problems: First, a priori, a ray may turn back to $\infty$. In \cite{rem2008}, the possibility of having the point at $\infty$ in the accumulation set of a ray when $t\rightarrow 0$ is excluded by using so called "extendibility" condition given by \cite[Defn 3.1, Observation 3.2]{rem2008} (see \cite[Thm B.2]{rem2008}). In our proof we rule out this possibility without examining extendibility condition, but with a more direct approach using hyperbolic geometry results. Second, if $f$ is not an exponential map, rays may be nonrectifiable. We treat this situation alternatively, by replacing the hyperbolic length of the fundamental segment with the hyperbolic length of another arc which has some particular properties.

In the proof, we are going to perform successive pull-backs of fundamental segments on a periodic ray, and we shall see that in the limit, these pullbacks degenerate to a single point. We would like to emphasize that working with maps with bounded post-singular sets has obvious advantages: Since the point on the rays escape to infinity under iteration, the postsingular set cannot intersect, or accumulate on the set of rays. This allows us to pull back a fundamental segment all the way along a periodic ray. Furthermore, having a bounded post-singular set provides useful hyperbolic geometry properties to study the dynamics in a neighborhood of $\infty$.

The idea of the proof of the Main Theorem is as follows: Suppose $g:(0,\infty)\rightarrow\mathbb{C}$ is a $k$-periodic dynamic ray. Let $W$ denote the accumulation set of $g$ as $t\rightarrow 0$. The goal is to show that $W$ is a singleton in the complex plane. Let $w$ be an arbitrary point in $W$, and $\{t_n\}_n$ be a sequence of positive numbers converging to $0$, such that $\lim_{n\rightarrow \infty}g(t_n)=w$. We show that the limit point $w$ is finite, and is a periodic point, i.e., $f^k(w)=w$. This holds for any accumulation point in $W$, which means $W$ is a discrete set. On the other hand, $W$ is connected, being an accumulation set of a connected set. Hence $W=\{w\}$, as we wanted to show. Suppose $U_0$ is the unbounded connected component of $\mathbb{C}\backslash\mathcal{P}$, and hence  $g\subset U_0$. We divide the proof into two cases according to whether the accumulation point $w$ is in $U_0$, or in $\partial U_0$. The possibility of having the point at infinity  in the accumulation set is ruled out under the case when $w\in\partial U_0$.
\subsubsection*{Acknowledgements} I would like to thank Carsten Lunde Petersen, for guiding throughout this work,  and N\'uria Fagella for many useful comments. This work was supported by Marie Curie RTN 035651-CODY and Roskilde University.
\section{Preliminaries on hyperbolic geometry}\label{sectionhyp}
This section contains the basic notions and results of hyperbolic geometry, which are the main tools in the proof of the Main Theorem. Proofs of known results can be found in, for example \cite{bearmin2000}, while the proof of results that are new or could not be found in the literature are included here.

The \textit{\textit{hyperbolic metric (Poincar\'{e} metric)}} in the unit disk $\mathbb{D}$ is defined by
\begin{equation*}
\lambda_{\mathbb{D}}(z)|dz|=\frac{2|dz|}{1-|z|^2},
\end{equation*}
which has constant curvature $-1$, where $\lambda_{\mathbb{D}}$ is called the \textit{hyperbolic density function}. This metric is invariant under all Mobius transformations $z\mapsto e^{i\theta}\frac{z-a}{1-\overline{a}z}$, $\theta\in\mathbb{R}$, $a\in\mathbb{D}$. More precisely, for any given Mobius transformation $M:\mathbb{D}\rightarrow\mathbb{D}$ and for all $z\in\mathbb{D}$,
\begin{equation*}
 \lambda_{\mathbb{D}}(f(z))|f'(z)|=\lambda_{\mathbb{D}}(z)
 \end{equation*}
 is satisfied. In other words, all Mobius transformations  are \textit{isometries} with respect to the hyperbolic metric $\lambda_{\mathbb{D}}(z)|dz|$.

By the Uniformization Theorem, we have the following:
\begin{thm}\label{thmhyperbolicmetric}
Every domain $U$ of the Riemann sphere $\widehat{\mathbb{C}}:=\mathbb{C}\cup\{\infty\}$ with at least three boundary points admits a hyperbolic metric. More precisely, there exists a unique hyperbolic metric $\lambda_{U}(w)|dw|$ on $U$ such that for any universal covering $f:\mathbb{D}\rightarrow U$, and for all $w=f(z)\in U$,
\begin{equation*}
\lambda_{\mathbb{D}}(z)=\lambda_{U}(f(z))|f'(z)|.
\end{equation*}
The metric $\lambda_{U}(w)|dw|$ is real analytic with curvature $-1$.
\end{thm}
Here $\lambda_{U}$ is the hyperbolic density function in the hyperbolic metric defined in $U$. The metric on $U$ induces a hyperbolic distance $d_{U}$ between two points $z$, $w$ in $U$ in the following way: the hyperbolic length of an arc $\gamma\in U$ joining $z$ and $w$ in $U$ is defined by
\begin{equation*}
l_{\gamma}=\int_{\gamma}\lambda_{U}(z)|dz|.
\end{equation*}
 Then the distance is equal to
 \begin{equation*}
  d_{U}(z,w)=\inf_{\gamma\in U} l_{\gamma}.
\end{equation*}
Holomorphic maps have the property that they do not increase the hyperbolic metric. We state this result below.
\begin{thm}(Schwarz-Pick Lemma)
Let $U$ and $V$ be hyperbolic domains  and suppose $f:V\rightarrow U$ is a holomorphic map. Then one of the following is satisfied:
\begin{itemize}
\item[i.] $f$ is a covering map. Hence it is a local hyperbolic isometry with respect to the two metrics, i.e.,
\begin{equation*}
\lambda_{V}(z)=|f'(z)|\lambda_{U}(f(z)). 
\end{equation*}
In this case, the hyperbolic arc lengths are preserved for the respective hyperbolic metrics. 
 
\item[ii.] $f$ decreases the hyperbolic metric, i.e.,

\begin{equation*}
|f'(z)|\lambda_{U}(f(z))<\lambda_{V}(z).
\end{equation*}
\end{itemize}

\end{thm}
For a given pair of hyperbolic domains $U$ and $V$ such that $V\subset U$, assigning the identity map 
\begin{equation*}
Id:V\rightarrow U,
\end{equation*}
Schwarz-Pick Lemma gives the Comparison Principle, which we state below.
\begin{thm} (Comparison Principle)
Let $U$ and $V$ be two hyperbolic domains such that  $V\subset U$. Then $\lambda_{V}(z)\geq\lambda_{U}(z)$ for all $z\in V$. 
\end{thm}
Not for all hyperbolic domains,  hyperbolic density functions can be calculated. Instead, one can use estimates, using the Comparison Principle. With this idea, for hyperbolic domains $U$ and $V$, such that $V\subset U\subsetneq\mathbb{C}$, we can find an upper bound for the ratio $\frac{\lambda_U}{\lambda_V}$, locally, as we show in the following.
\begin{lem}\label{contractiongeneral} Let  $V, U\subsetneq\mathbb{C}$ be hyperbolic domains such that $V\subset U$. Set $d:=d_U(z,\partial V)$. Then, for all $z\in V$,
\begin{equation*}
0<\frac{\lambda_U(z)}{\lambda_V(z)} \leq \kappa(d),\;\;\;\;\mathrm{where}
\end{equation*}
\begin{equation}\label{contractionestimate}
\kappa(d):=-\frac{e^{2d}-1}{2 e^{d}}\log(\frac{e^d-1}{e^d+1}),\;\;\mathrm{with}\;\;\kappa(0)=0.
\end{equation}

\end{lem}
\begin{proof} Let $z\in V$ be an arbitrary point, and let $\pi:\mathbb{D}\rightarrow U$ be a universal covering, such that $\pi(0)=z$. Let $\widehat{V}$ be the connected component of $\pi^{-1}(V)$ containing $0$. The map $\pi:\mathbb{D}\rightarrow U$ and the restriction $\pi|_{\widehat{V}}:\widehat{V}\rightarrow V$ are local hyperbolic isometries, that is,
\begin{eqnarray*}
\lambda_{\mathbb{D}}(0)=|\pi'(0)|\lambda_U(z),\\
\lambda_{\widehat{V}}(0)=|\pi'(0)|\lambda_V(z),
\end{eqnarray*}
and hence
\begin{equation*}
\frac{\lambda_{\mathbb{D}}(0)}{\lambda_{\widehat{V}}(0)}=\frac{\lambda_U(z)}{\lambda_V(z)}.
\end{equation*}
Moreover, $d_{\mathbb{D}}(0,\partial\widehat{V})=d_U(z,\partial V)$.
For $x\in\mathbb{D}\backslash \widehat{V}$, set $V_x=\mathbb{D}\backslash\{x\}$. Since by the Comparison Principle $\widehat{V}\subset V_{x}$, $\lambda_{V_x}(z)\leq \lambda_{\widehat{V}}(z)$ for all $z\in\widehat{V}$, and hence 
\begin{equation}\label{righthandsideoftheinequality}
\frac{\lambda_{\mathbb{D}}(0)}{\lambda_{\widehat{V}}(0)}\leq\frac{\lambda_{\mathbb{D}}(0)}{\lambda_{V_x}(0)}.
\end{equation}
The M\"{o}bius transformation $M:\mathbb{D}\rightarrow\mathbb{D}$, $\displaystyle{M(w)=\frac{w-x}{1-\overline{x}w}}$ is a hyperbolic isometry with $M(0)=-x$, so that we have $\lambda_{\mathbb{D}}(0)=\lambda_{\mathbb{D}}(-x)|M'(0)|$, and $\lambda_{V_x}(0)=\lambda_{\mathbb{D}^*}(-x)|M'(0)|$. Then
\begin{equation}\label{kappa}
\frac{\lambda_{\mathbb{D}}(0)}{\lambda_{\widehat{V}}(0)}\leq\frac{\lambda_{\mathbb{D}}(0)}{\lambda_{V_x}(0)}=\frac{\lambda_{\mathbb{D}}(-x)}{\lambda_{\mathbb{D}^*}(-x)}=-\frac{2|x| \log| x|}{1-|x|^2}.
\end{equation}
(Here $\lambda_{\mathbb{D}^*}(z)=-\frac{1}{|z|\log|z|}$. One can obtain this using Theorem \ref{thmhyperbolicmetric}). Note that $-\frac{2 |x|\log |x|}{1-|x|^2}$ is increasing for $|x|\in(0,1)$.

We have the relation
\begin{equation*}
d:=d_U(z,\partial V)\leq d_{\mathbb{D}}(0,-x)=\log\frac{1+|x|}{1-|x|},
\end{equation*}
or equivalently
\begin{equation}\label{modx}
|x|\leq \frac{e^d-1}{e^d+1},
\end{equation}
since $\log\frac{1+|x|}{1-|x|}$ is increasing in $|x|\in(0,1)$. Using (\ref{modx}) in (\ref{kappa}), we obtain the desired result.
\end{proof}
\begin{lem}\label{lem3}
Let $U\subset \widehat{\mathbb{C}}$ be a hyperbolic domain. Suppose that $\{x_n\}_n$ and $\{y_n\}_n$ are sequences in $U$ such that $d_U(x_n,y_n)\leq \kappa<\infty$. If $x_n\rightarrow x\in\partial U$, then $y_n\rightarrow x$. 
\end{lem}
For the general proof of Lemma \ref{lem3} for Riemann surfaces, see for example \cite[Thm 3.4]{mil2006}.

The lemma below is going to be used in the proof of Proposition \ref{mainprop}, in order to simplify the domain we work on.  
\begin{lem}\label{lem_covering}
Let $U\subseteq\widehat{\mathbb{C}}$ be a hyperbolic domain with an isolated boundary point $z_0$, and let $\widehat{U}:=U\cup\{z_0\}$ be a domain. Then there exists a holomorphic map $\pi_*:\mathbb{D}\rightarrow \widehat{U}$, with $\pi_*(0)=z_0$, $\pi_*'(0)\neq 0$, and  for $\mathbb{D}^*:=\mathbb{D}\backslash\{0\}$, the restriction $\pi_{*}|_{\mathbb{D}^*}:\mathbb{D}^*\rightarrow U$  is a covering map of degree $1$ in a neighborhood of  $0$.
\end{lem}

Note that given hyperbolic domains $U$ and $V$ such that $V\subset U$, according to Lemma \ref{contractiongeneral}, an upper bound for  $\frac{\lambda_U}{\lambda_V}$ is calculated only for points inside $V$. Assuming $U$ and $V$ have a common isolated boundary point $z_0$, the next proposition gives a condition for which $\frac{\lambda_U}{\lambda_V}$ can be estimated from above in any neighborhood of $z_0$.
\begin{prop}\label{mainprop}
Let $U$ be a hyperbolic domain with an isolated boundary point $z_0$. Let $\{w_n\}_n$ be a sequence in $U$ such that $w_n\rightarrow z_0$ as $n\rightarrow \infty$ and that $d_{U}(w_n,w_{n+1})\leq\delta$. Let $V=U\backslash\{w_n\}_n$. Given a simply connected and relatively compact neighborhood $\widetilde{\Omega}$ of $z_0$ in $U\cup\{z_0\}$, there exists $\kappa>0$, such that
\begin{equation*}
\forall z\in V\cap \widetilde{\Omega},\;\;\;\;0<\frac{\lambda_{U}(z)}{\lambda_V(z)}\leq\kappa<1.
\end{equation*}
\end{prop}

\begin{proof}
We will use the contraction estimate given by (\ref{contractionestimate}) in Lemma \ref{contractiongeneral}. In order to do so, first we need to show for all $z\in V\cap \widetilde{\Omega}$, 
\begin{equation*}
d_{U}(z,\partial V)=d_{U}(z,\{w_n\}_n)<\infty.
\end{equation*}
There exists a holomophic covering map $\pi_{*}:\mathbb{D}^*\rightarrow U$, which has local degree $1$ in a neighborhood $\widehat{\Omega}\in\mathbb{D}$ of $0$ such that $\pi_{*}(0)=z_0$, by Lemma \ref{lem_covering}. The map $\pi_{*}$ is a local isometry  with respect to the hyperbolic metrics in $\mathbb{D}^*$ and $U$.  For points $\widehat{w}_n$ in $\widehat{\Omega}$ near $0$, with $\pi_{*}(\widehat{w}_n)=w_n\in\widetilde{\Omega}$, we have 
\begin{equation}\label{Eq}
d_{U}(w_n,w_{n+1})=d_{\mathbb{D}^*}(\widehat{w}_n,\widehat{w}_{n+1})\leq\delta.
\end{equation}
Let $\{\widehat{w}_n\}_n\subset\widehat{\Omega}$ be as above, and $\widehat{z}\in\widehat{\Omega}$ be such that $\pi_{*}(\widehat{z})=z$.  Then the problem reduces to proving
\begin{equation}\label{distanceinpunctureddisk}
\displaystyle{\sup_{\hat{z}\in\widehat{\Omega}}d_{\mathbb{D}^*}(\widehat{z},\{\widehat{w}_n\}_n)<\infty}.
\end{equation}
For any $0<r<R<1$, $A:=\widehat{\Omega}\cap\{\widehat{z}:\;\;r\leq |\widehat{z}|\leq R\}$  is compactly contained in $\mathbb{D}^*$. So the Euclidian and the hyperbolic distances are comparable in $A$, hence (\ref{distanceinpunctureddisk}) holds in $A$.

We need to prove this also holds in $\mathbb{D}^*(0,r):=\{\widehat{z}:\;0<|\widehat{z}|<r\}$ for small $r$ values. We estimate $d_{\mathbb{D}^*}(\widehat{z},\{\widehat{w}_n\}_n)$ in the following way: We take the sequence of circles $\{\mathbb{S}(0,r_n)\}_n$ centered at $0$ with Euclidian radius $r_n$, passing through the points of $\{\widehat{w}_n\}_n$ in $\mathbb{D}^*(0,r)$ (i.e., $\widehat{w}_n\in \mathbb{S}(0,r_n))$.  The hyperbolic length of $\mathbb{S}(0,r_n)$ in $\mathbb{D}^*$ is equal to $l_{\mathbb{D^*}}(\mathbb{S}(0,r_n))=-\frac{2\pi}{\log r_n}$. Observe that $l_{\mathbb{D^*}}(\mathbb{S}(0,r_n))\rightarrow 0$ as $r_n\rightarrow 0$. We will use the arcs on the circles as paths connecting the points of the sequence $\{\widehat{w}_n\}_n$, to the initial point $\widehat{z}$. Surely the hyperbolic length of these arcs is larger than $d_{\mathbb{D}^*}(\widehat{z},w_n)$.
\\
\\
\begin{figure}[htb!]
\begin{center}
\def\svgwidth{6 cm}
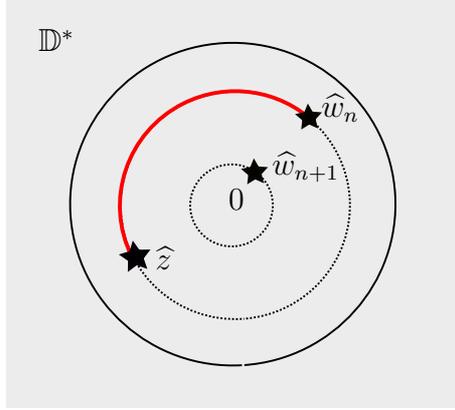
\caption{When $\widehat{z}$ is on the circle $\mathbb{S}(0,r_n)$.}\label{one}
\end{center}
\end{figure}

 Observe
\begin{itemize}
\item[i.] For any $\widehat{z}\in \mathbb{S}(0,r_n)$, $d_{\mathbb{D}^*}(\widehat{z},\widehat{w}_n)<-\frac{\pi}{\log r_n}$ (see Figure \ref{one}). 

\item[ii.] For a point $\widehat{z}$ in the annulus bounded by $\mathbb{S}(0,r_n)$ from outside and $\mathbb{S}(0,r_{n+1})$ from inside, $d_{\mathbb{D}^{*}}(\widehat{z},\widehat{w}_n)<\delta-\frac{\pi}{\log r_n}$, since $d_{\mathbb{D}^*}(\widehat{w}_n,\widehat{w}_{n+1})\leq\delta$ (see (\ref{Eq}) and Figure \ref{two}).
\end{itemize}
\vspace{5mm}

\begin{figure}[htb!]
\begin{center}
\def\svgwidth{6 cm}
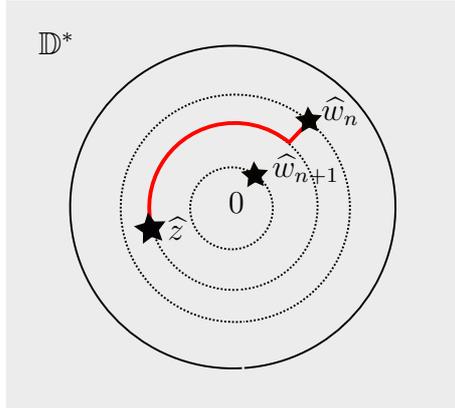
\caption{When $\widehat{z}$ is inside the annulus bounded by the circles $\mathbb{S}(0,r_n)$ and $\mathbb{S}(0,r_{n+1})$.}\label{two}
\end{center}
\end{figure}

We see that in both cases there exists $d>0$ such that $d_{\mathbb{D}^{*}}(\widehat{z},\{\widehat{w}_n\}_n)\leq d<\infty$. This means $d_{U}(z,\{w_n\}_n)\leq d<\infty$.

In this case, it is possible to find an upper bound $\kappa=\kappa(r_n,\delta)$. Depending on the position of $z$, we obtain 
\begin{equation*}
\frac{\lambda_{U}(z)}{\lambda_{V}(z)}\leq\kappa(r_n,\delta)=-\frac{e^{2d}-1}{2 e^{d}}\log\big(\frac{e^{d}-1}{e^{d}+1}\big)<1,
\end{equation*}
where $d:=\delta-\frac{\pi}{\log r_n}$, (case $i.$ for $\delta=0$) for all $z$ in a neighborhood  of $z_0$.
\end{proof}
\begin{rem}
Let $U$, $V$ be as in the statement of Proposition \ref{mainprop}. For any hyperbolic domain $W\subset V$, $0<\frac{\lambda_U(z)}{\lambda_W(z)}\leq\kappa<1$ holds in a neighborhood of $z_0$, by the Comparison Principle.
\end{rem}

\section{Setup}\label{sectionsetup}

Let $f$ and $\mathcal{P}$ be as in the statement of the Main Theorem. Take a $k$-periodic dynamic ray $g$. We denote by $U_0$, the unbounded connected component of $\mathbb{C}\backslash \mathcal{P}$, and by $U_1$, the connected component of $f^{-k}(U_0)$, which contains $g$. Since $\mathcal{P}$ contains at least two points, $U_0$ is hyperbolic. Since the restriction $f^k|_{U_1}:U_1\rightarrow U_0$ is a covering, it is a local hyperbolic isometry, that is, for the respective hyperbolic density functions $\lambda_{U_0}$ and $\lambda_{U_1}$, and for all $z\in U_1$,
\begin{equation*}
\lambda_{U_1}(z)|dz|=\lambda_{U_0}(f^k(z))|d(f^k(z))|=\lambda_{U_0}(f^k(z))|(f^k)'(z)||dz|
\end{equation*} 
is satisfied. In other words, $f^k$ preserves the hyperbolic arc lengths. Furthermore since the post-singular set is not backward invariant,  $U_1\subsetneq U_0$. Thus by the Comparison Principle,
\begin{equation*}
 \frac{\lambda_{U_0}(z)}{\lambda_{U_1}(z)}<1,\;\;\;\;\forall z\in U_1.
 \end{equation*}
 Let $F:[0,\infty)\rightarrow[0,\infty)$ denote a ray model dynamics as described in the introduction. Recall that for a $k$-periodic ray $f^k(g(t))=g(F^k(t))$. We denote a ray segment with end points $g(r_1)$ and $g(r_2)$ for $r_1,r_2\in (0,\infty)$, by $g[r_1,r_2]$.
Recall that a fundamental segment of the periodic ray $g$ is a ray segment with endpoints $g(t)$ and $g(F^k(t))$  for $t>0$, and we denote it by $g[t, F^k(t)] $.
\subsection{Study of nonrectifiable curves}
 It is not yet known for class $\widehat{\mathcal{B}}$, whether dynamic rays are rectifiable. Hence, a priori, fundamental segments in our setting may have infinite hyperbolic length. Let us denote the hyperbolic arc lengths in $U_0$ and $U_1$ by $l_{U_0}$ and $l_{U_1}$, respectively. Let $\Gamma[r_1,r_2]$ be the set of curves in $U_0$ with endpoints $g(r_1)$ and $g(r_2)$, which are in the same homotopy class with $g[r_1,r_2]$ relative to the endpoints. Take $\gamma^{*}[r_1,r_2]\in\Gamma[r_1,r_2]$ such that
\begin{equation*}
l_{U_0}(\gamma^{*}[r_1,r_2])=\inf_{\gamma\in\Gamma[r_1,r_2]}\int_{\gamma}\lambda_{U_0}(z)|dz|.
\end{equation*}
Note that $\gamma^{*}[r_1,r_2]$ is the unique geodesic connecting $g(r_1)$ and $g(r_2)$ in the same homotopy class with $g[r_1,r_2]$. Now define
\begin{equation*}
diam^*_{U_0} (g[t, F^k(t)]):=\sup_{t\leq r_1\leq r_2\leq F^k(t)}l_{U_0}(\gamma^{*}[r_1,r_2]).
\end{equation*}
In this setting, $diam^*_{U_0}(g[t,F^k(t)])$ is greater or equal to the hyperbolic diameter of $g[t,F^k(t)]$ in $U_0$.


\begin{lem}\label{nondecreasingsequence}
The map $M:(0,\infty)\rightarrow(0,\infty)$
\begin{equation*}
M(t):=\sup_{\tau\in[t,F^k(t)]}diam^*_{U_0}(g[\tau,F^k(\tau)])
\end{equation*}
 is nondecreasing.
\end{lem}

\begin{proof}

Observe that $M<\infty$, because the hyperbolic distance between any two points is bounded on $g[t,F^{2k}(t)]$, $t>0$, since $g[t,F^{2k}(t)]$ is compact. Take $t_1,t_2\in(0,\infty)$ such that $t_1<t_2$.
\begin{itemize}
\item[i.]  Suppose $F^k(t_1)<t_2$. Take  $g(r_1)\in g[t_1,F^{k}(t_1)]$ and $r_2$, such that
\begin{equation*}
l_{U_0}(\gamma^{*}[r_1,r_2])=M(t_1).
\end{equation*}
Observe that $g(r_2)\in g[t_1,F^{2k}(t_1)]$.
The point $g(r_1)$ eventually maps into $g[t_2,F^k(t_2)]$, say, after $nk$ iterates, i.e., 
\begin{equation*}
g(F^{nk}(r_1))\in g[t_2,F^k(t_2)].
\end{equation*}
In this case, 
\begin{equation*}
g(F^{nk}(r_2))\in g[t_2,F^{2k}(t_2)].
\end{equation*}
Take the curve $\gamma^{*}[F^{nk}(r_1),F^{nk}(r_2)]$. Obviously
\begin{equation}\label{geodesicinequality1}
l_{U_0}(\gamma^{*}[F^{nk}(r_1),F^{nk}(r_2)])\leq M(t_2).
\end{equation}
by hypothesis.

Let $\gamma^{i}[F^{(n-i)k}(r_1),F^{(n-i)k}(r_2)]$ be the lift of $\gamma^{*}[F^{nk}(r_1),F^{nk}(r_2)]$ by $f^{ik}$, with  endpoints $g(F^{(n-i)k}(r_1))$ and $g(F^{(n-i)k}(r_2))$. Observe that 
\begin{equation*}
\gamma^{i}[F^{(n-i)k}(r_1),F^{(n-i)k}(r_2)]\in\Gamma[F^{(n-i)k}(r_1), F^{(n-i)k}(r_2)]. 
\end{equation*} 
 This is because $f^{ik}$ is a holomorphic covering in $U_0$, and hence it has the homotopy lifting property. Now take $i=1$. The map $f^k|_{U_1}:U_1\rightarrow U_0$ is a local hyperbolic isometry, that is,
\begin{equation}\label{aa}
l_{U_0}(\gamma^{*}[F^{nk}(r_1),F^{nk}(r_2)])=l_{U_1}(\gamma^1[F^{(n-1)k}(r_1), F^{(n-1)k}(r_2)])
\end{equation}
holds. Moreover, since $U_1\subset U_0$, by the Comparison Principle,
\begin{equation}\label{bb}
l_{U_0}(\gamma^1[F^{(n-1)k}(r_1), F^{(n-1)k}(r_2)])<l_{U_1}(\gamma^1[F^{(n-1)k}(r_1), F^{(n-1)k}(r_2)]),
\end{equation}
and hence by (\ref{aa}) and (\ref{bb})
\begin{equation*}
l_{U_0}(\gamma^1[F^{(n-1)k}(r_1), F^{(n-1)k}(r_2)])<l_{U_0}(\gamma^{*}[F^{nk}(r_1),F^{nk}(r_2)]).
\end{equation*}
Applying pull-back $n$ times under $f^k$, we obtain:
\begin{equation}\label{geodesicinequality3}
l_{U_0}(\gamma^n[r_1,r_2])<l_{U_0}(\gamma^{*}[F^{nk}(r_1),F^{nk}(r_2)])
\end{equation}
Observe that
\begin{equation}\label{geodesicinequality2}
M(t_1)= l_{U_0}(\gamma^{*}[r_1,r_2)]\leq l_{U_0}(\gamma^n[r_1,r_2]).
\end{equation}
By (\ref{geodesicinequality1}), (\ref{geodesicinequality3})and (\ref{geodesicinequality2}), we obtain
\begin{equation}\label{M_1}
M(t_1)<M(t_2).
\end{equation}
\item[ii.] If $t_2<F^k(t_1)$, there may be an exceptional case, such that 
\begin{equation}\label{M_2}
M(t_1)=M(t_2)=
\sup_{\tau\in[t_1,F^k(t_2)]}diam^*_{U_0}(g[\tau,F^k(\tau)])
\end{equation}
Thus  (\ref{M_1}) and (\ref{M_2}) complete the proof.
\end{itemize}
\end{proof}

\section{Proof of the Main Theorem}
Let $W$ denote the accumulation set of the $k$-periodic dynamic ray $g$ as $t\rightarrow 0$. We divide the proof into two cases, according to whether an accumulation point $w\in W$ is in $U_0$, or in $\widehat{\mathbb{C}}\backslash U_0$. Recall that $U_0$ is the unbounded connected component of $\mathbb{C}\backslash \mathcal{P}$, where $\mathcal{P}$ is the post-singular set, $U_1$ is the set for which $f^k|_{U_1}:U_1\rightarrow U_0$ is a covering, and which contains $g$.
\subsection{When the accumulation point  is in $U_0$}

Suppose  $w\in U_0$ is an accumulation point of $g$ as $t\rightarrow 0$. Take a point $t_0\in\mathbb{R}$, such that $d_{U_0}(g(t_0),w)<N<\infty$. Take a sequence $\{t_n\}_n$, starting from $t_0$, such that $g(t_n)\rightarrow w$, and $d_{U_0}(g(t_n),w)<N$. Write $M:=M(t_0)$. Then, all double-fundamental segments $g[t_n, F^{2k}(t_n)]$ are going to be in the closed hyperbolic ball $\Omega:=\overline{\mathbb{D}_{U_0}(w,2M+N)}$ of center $w$ and radius $2M+N$ by Lemma \ref{nondecreasingsequence} (see Figure \ref{figure1}). 

\begin{figure}[htb!]
\begin{center}
\def\svgwidth{8 cm}
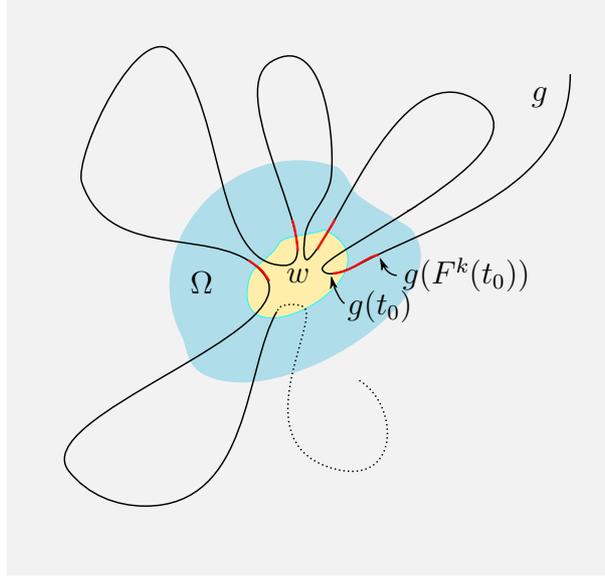
\caption{\small{All fundamental segments, starting at $g[t_n,F^k(t_n)]$, $n\geq 0$ are inside $\Omega$. In order to have a clear picture, the dynamic ray $g$ is assumed to be smooth.}}
\label{figure1}
\end{center}
\end{figure}

Assume
\begin{equation*}
F^{2k}(t_{n+1})<t_n
\end{equation*}
for all $n\in\mathbb{N}$, passing to a subsequence, if necessary. Since $\Omega$ is compact and the contraction ratio $\kappa(z):=\frac{\lambda_{U_0}(z)}{\lambda_{U_1}(z)}$ is continuous, the function $\kappa(z)$ attains its maximum value at some point $z_0\in\Omega$, i.e., for all $z\in\Omega$, $\kappa(z)\leq\kappa:=\kappa(z_0)<1$. This yields, 
\begin{equation*}
M(t_{n+1})\leq\kappa M(t_n)\leq \kappa^{n+1}M(t_0)\rightarrow 0, \;\;\mathrm{as}\;\;n\rightarrow\infty
\end{equation*}

(see Lemma \ref{nondecreasingsequence}). This implies
\begin{equation*}
diam^*_{U_0}(g[t_n,F^k(t_n)])\rightarrow 0, \mathrm{as}\;\; n\rightarrow \infty,
\end{equation*}
which means $d_{U_0}(g[t_n,F^k(t_n)])\rightarrow 0$, and hence $w=f^k(w)$ for the limit point $w$.
\subsection{When the accumulation point  is not in $U_0$}

In this case $w\in\partial U_0$.  First, we are going to show that $\infty$ is not an accumulation point of a periodic dynamic ray as $t\rightarrow 0$.
\subsubsection{Study in a neighborhood of \texorpdfstring{$\infty$}{infinity}}
Our goal is to use Proposition \ref{mainprop} in order to conclude that the inverse branches of $f^k$ are strongly contracting with respect to the hyperbolic metric in $U_0$ in a neighborhood of $\infty$. In that case, the iterative pull-backs of $\gamma^*[t_0,F^k(t_0)]$ along the ray will be bounded by   
\begin{equation*}
\sum_{n=0}^{\infty} \kappa^n M=\frac{M}{1-\kappa}<\infty,
\end{equation*}
that is,
\begin{equation*}
\sum_{n=0}^{\infty}l_{U_0}(\gamma^*[t_n,F^k(t_n)])<\frac{M}{1-\kappa}<\infty.
\end{equation*}
So the existence of such $\kappa$ in a neighborhood of $\infty$ guarantees that the periodic ray sufficiently close to $\infty$ for small potentials lands before reaching $\infty$.

We will check that the hypothesis of Proposition \ref{mainprop} are satisfied. We take the isolated boundary point $z_0=\infty$. We need to show the existence of a sequence $\{w_i\}_i\subset U_0\backslash U_1$ such that $w_i\rightarrow\infty$ as $i\rightarrow\infty$ with $d_{U_0}(w_i,w_{i+1})\leq \delta$.  
This is given by the following lemma.

\begin{figure}[htb!]
\begin{center}
\def\svgwidth{8 cm}
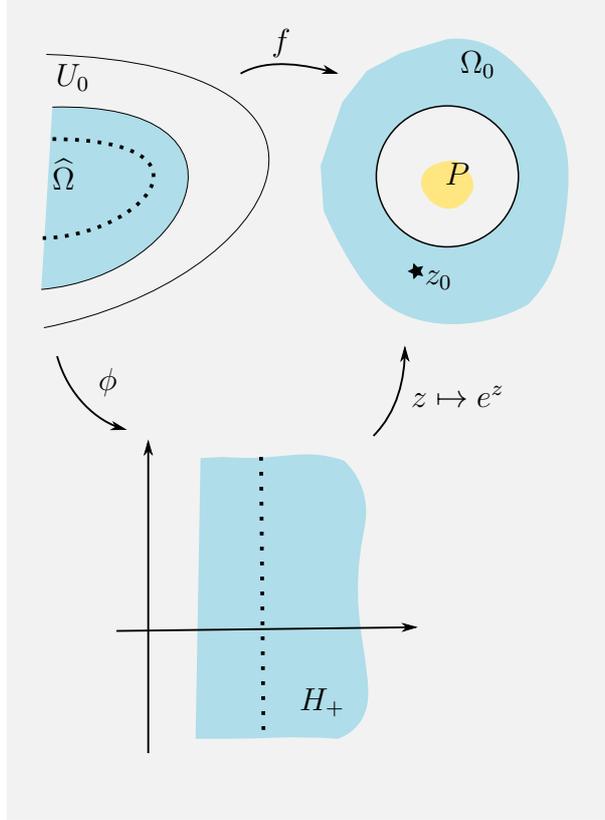
\caption{\small{The preimages of $z_0$ under $f$  are tending to $\infty$ in $\widehat{\Omega}$.}}
\label{figure}
\end{center}
\end{figure}
\begin{lem}\label{lemmainaneighbor}
$U_0\backslash U_1$ contains a sequence of points $\{w_i\}_{i\in\mathbb{Z}}$, such that  $w_i\rightarrow \infty$ as $i\rightarrow\pm\infty$ with $d_{U_0}(w_i,w_{i+1})<\delta$ for some $\delta>0$.
\end{lem}

\begin{proof} Let $\mathbb{D}(0,R)$ denote the disk with radius $R$, centered at $0$, such that $\mathcal{P}\subset \overline{\mathbb{D}(0,R)}$, and set $\Omega_0:=\mathbb{C}\backslash\overline{\mathbb{D}(0,R)}$. Then there exists a simply connected domain  $\widehat{\Omega}\subset \mathbb{C}$, and the restriction $f|_{\widehat{\Omega}}:\widehat{\Omega}\rightarrow\Omega_0$ is a universal covering (such $\widehat{\Omega}$ is called an exponential tract). Let $z_0\in \Omega_0$ such that $f^{k-1}(z_0)\in \mathcal{P}$. The pre-images $\{w_i\}_{i\in\mathbb{Z}}$ of $z_0$ under $f$ tends to $\infty$ as $i\rightarrow\pm\infty$. We claim that, it is a sequence of equally spaced points, with respect to the hyperbolic metric in $\widehat{\Omega}$. Indeed, since $\widehat{\Omega}$ is simply connected, there exists a biholomorphic map $\phi:\widehat{\Omega}\rightarrow H_{+}$ to some right half plane $H_{+}$ such that $\exp\circ\phi=f$ (see Figure \ref{figure}). Let $\{\widehat{w}_i\}_{i\in\mathbb{Z}}$ be such that $e^{\widehat{w}_i}=z_0$. Observe that the consecutive elements of $\{\widehat{w}_i\}_{i\in\mathbb{Z}}$ are $2\pi i$ translates of each other in $H_{+}$. Set $\delta:=d_{H_{+}}(\widehat{w}_i,\widehat{w}_{i+1})$, where $d_{H_{+}}$ denotes the hyperbolic distance in $H_{+}$. Observe that $\delta$  is independent of $i$. Write $w_i=\phi^{-1}(\widehat{w}_i)$. Since $\phi^{-1}$ is conformal, it preserves the respective hyperbolic distances, i.e.,
\begin{eqnarray*}
\delta=d_{H_{+}}(\widehat{w}_i,\widehat{w}_{i+1})&=&d_{\widehat{\Omega}}(\phi^{-1}(\widehat{w}_i),\phi^{-1}(\widehat{w}_{i+1})\notag\\
&=&d_{\widehat{\Omega}}(w_i,w_{i+1}).
\end{eqnarray*}
Since $\widehat{\Omega}\subset U_0$, this implies $d_{U_0}(w_i, w_{i+1})< d_{\widehat{\Omega}}(w_i,w_{i+1})=\delta$, by the Comparison Principle.
\end{proof}
With Lemma \ref{lemmainaneighbor}, we conclude that all accumulation points of periodic rays as $t\rightarrow 0$ are finite. If $w$ is a boundary point of $U_0$, then it is a common boundary point with $U_1$. 
Let $g(t_0)\in U_0$. Recall 
\begin{equation*}
M(t_0)=\sup_{\tau\in[t_0,F^k(t_0)]}diam^*_{U_0}(g[\tau, F^k(\tau)]). 
\end{equation*}
Take a sequence $\{t_n\}_n$ such that $t_n\rightarrow 0$ and $g(t_n)\rightarrow w\in\partial U_0$ as $n\rightarrow \infty$. Then by Lemma \ref{nondecreasingsequence},
\begin{equation*}
diam^*_{U_0}(g[t_n,F^k(t_n)])\leq M(t_0)
\end{equation*}
for all $t_n<t_0$, $n>0$. By Lemma \ref{lem3}, this implies that $g(F^k(t_n))\rightarrow w$, since
 $g(t_n)\rightarrow w\in\partial U_0$. Hence for the limit point $w$, the equality $w=f^k(w)$ is satisfied.

\subsection{Conclusion of the proof of the Main Theorem}

To summarize, no matter whether the accumulation point $w$ is in $U_0$, or not, it satisfies  $w=f^k(w)$. To see that this is the unique limit point, take another accumulation point $w'$ of $g$ as $t\rightarrow 0$. Repeating the same argument, we obtain $w'=f^k(w')$. But the solution set of this equation is discrete while the limit set is connected, as being the limit set of a connected set. Therefore the accumulation set consists of only one point. By the Snail Lemma (see, for example \cite[Lem 16.2]{mil2006}), $w$ is either a parabolic or a repelling periodic point of period $k$. In particular, in case $w\in U_0$, since the  inverse branches of $f^k$ are strongly  contracting with respect to the hyperbolic metric in $U_0$, the limit point is repelling. This completes the proof.


\bibliographystyle{plain}
\end{document}